\documentclass[11pt]{article}
\usepackage{amsmath,amssymb,latexsym}
\usepackage{graphics}
\usepackage{amscd}
\usepackage[latin1]{inputenc}
\usepackage[T1]{fontenc}
\usepackage[dvips]{graphicx}
\usepackage{rotating,amssymb,amsmath,amsfonts,delarray}
\usepackage{geometry}
\usepackage[usenames]{color}
\usepackage{placeins}
\usepackage{float}
 \makeatletter

\newtheorem{thm}{Theorem}[section]

\newtheorem{defn}{Definition}[section]

\newtheorem{lem}{Lemma}[section]

\newtheorem{prop}{Proposition}[section]
\newtheorem{rem}{Remark}[section]

\newenvironment{proof}[1][Proof]{\textbf{#1.} }{\ \rule{0.5em}{0.5em}}

\numberwithin{equation}{section}

\title{On a Generalization of the Expected Discounted Penalty Function to Include Deficits at and Beyond Ruin }

 \author{Zied Ben Salah \footnote{Mailing address: Zied Ben Salah. Department of Mathematics and Statistics. Université de Montréal. CP. 6128 succ. centre-ville. Montreal, Quebec. H3C 3J7. CANADA. Email: bensalah@dms.umontreal.ca} }
 \date{{\scriptsize First draft: Janvier~3, 2011. This version: \today.}}

\begin{document}

\maketitle

\begin{abstract}
{\scriptsize 
In this chapter we propose an extended concept of the expected discounted penalty function (EDPF)
that takes into account new ruin-related random variables. We add to the EDPF, which was introduced in classical papers
[Gerber and Shiu (1997), (1998) and Gerber and Landry (1998)], a sequence of expected discounted functions of new record minima reached by a jump of the risk process after ruin. Inspired by results of Huzak \emph{et al.} (2004) and developpements in fluctuation theory for spectrally negative Lévy processes, we provide a characterization for this extended EDPF in a setting involving a cumulative claims modelled by a subordinator, and Brownian perturbation. We illustrate how the extended EDPF can be used to compute the expected discounted value of capital injections  (EDVCI) for Brownian perturbed risk
model. 
\text{}
}
\end{abstract}

\section{Introduction}

\indent The concept of Expected Discounted Penalty Function (EDPF) has been introduced in classical papers [Gerber and Shiu (1997), (1998)]. This so-called Gerber-Shiu function is a functional of the ruin time (i.e., the first time the reserve level of a firm becomes negative), the surplus prior to ruin, and the deficit at ruin. The EDPF has been extensively studied and generalized to various scenarios and there is now a wide range of models for which expressions of the EDPF are available. 
Since the EDPF operates on a random cashflow at ruin, where the cashflow is a function of the deficit at ruin and the surplus prior to ruin, applications in the context of insurance and finance are quite natural. For example, the EDPF can be used to determine the initial capital required by an insurance company to avoid insolvency with a minimum level of confidence and for fixed penalization of the ruin event. Similary, the EDPF can be used as a pricing device for American options [Gerber and Shiu (1998b)].

In oder to have a more valuable Gerber-Shiu function for the management of insurance risks and the monitoring of the solvency of a firm, Biffis and Morales (2010) extended the EDPF to include path-dependent penalties. In particular, they generalized the definition of EDPF to include a new  random variable, the last minimum of the surplus before ruin. 
They obtained a defective renewal equation for this generalized EDPF. The representation is obtained for a subordinator risk model perturbed by a spectrally negative Lévy process. More generally, when the risk process is driven by a spectrally negative Lévy process, Biffis and Kyprianou (2010) provided an explicit characterization of this generalized EDPF in terms of scale functions, extending results available in literature. One of the reasons for the limited use of such EDPF is that the last minimum of the surplus before ruin, the surplus prior to ruin and the deficit at ruin, only characterize the surplus before and in a neighborhood of the ruin time. In other words, none of the arguments in the EDPF can be used as a predictive tool for successive deficit times after ruin. The situation would change if a penalty could apply after ruin, for example by acting on relevant characteristics of the paths of the risk process that may lead to successive minima after ruin, and not just on its level before, at, or immediately prior to ruin. 

In this chapter we show how to extend the EDPF to include these new random variables. In particular, we generalize the EDPF to include the sequence of successive record minima reached by a jump of the risk process after ruin. We obtain a new form of EDPF which gives characteristics of the paths of the risk process after ruin and not only before, and in a neighborhood  of, ruin time. There are practical applications of this extended EDPF in the context of insurance and reinsurance. For example, it can be used to determine the capital required by an insurance company to survive not only in the neighborhood of ruin, but also after ruin when the risk process continuous to jump downwards, that is it continues to pay out claims. Similary, this extended EDPF can be used by insurance company or government institutions to determine the capital which should be injected at each deficit time at and after ruin that will allow it to continue its operations. 
At the time of ruin, the insurance company could have access to other reserves allowing it to survive and pay to the customers claims made after ruin. It is at this time that the company might need the capital injections that will allow the net aggregate cash inflow to return to pre-ruin levels. 
This ruin cycle may occur several times and requires repeated interventions with capital injections. Thus, it would be interesting to add to the EDPF the expectation of a sequence of discounted functions of the successive minima reached by claims of the risk process after ruin. This extended EDPF is needed to determine the Expected Discounted Value of Capital Injection (EDVCI) for a subordinator risk model perturbed by a Brownian motion. Moreover, our approach provides a connection with some of the works in Einsenberg and Schmidli (2011) where they have studied the EDVCI for the classical risk model. Consequently, we generalize some results in Einsenberg and Schmidli (2011) where a similar problem is solved for the classical risk model. 

We use the results of Huzak \emph{et al.} (2004) and developments in fluctuation theory for spectrally negative Lévy processes to give an explicit characterization of this extended EDPF. The characterization is obtained for a subordinator risk model perturbed by a Brownian motion.     

The chapter is organized as follows. In Section 3.2, we describe a perturbed subordinator risk model and define our extended EDPF. In Section 3.3 and we review and provide some results that are needed in our derivations. In particular, we briefly review some results in Huzak \emph{et al.} (2004) about fluctuation theory for spectrally negative Lévy processes, and we give some preliminary results for first-passage times of a subordinator risk process perturbed by a Brownian motion under a change of measure. In section 4, we provide the expression of the extended EDPF in terms of convolution product of densities which are identified in Section 2. In section 5, we show  how the results of section 4 can be used to give explicitly the EDVCI for a subordinator risk model perturbed by a Brownian motion, and how it can be used to recuperate the expression of the EDVCI for the classical risk model [see Einsenberg and Schmidli (2011)]. Finally, Section 6 offers some concluding remarks.
\section{Risk model and the Expected Discounted Penalty Function}\label{riskEDPF}
Let $(\Omega, \mathbf{F}, \mathbb{P})$ be a filtered probability space on which all random variables will be defined. Let us define $S=(S_t)_{ t\geq 0}$ to be a subordinator (i.e., a Lévy process of bounded variation and nondecreasing paths) without a drift. Let $\nu$ be the Lévy measure of $S$; that is, $\nu$ is a $\sigma$-finite measure on $(0, \infty)$ satisfying $ \int_{(0,\infty)}(1\wedge y )\nu(dy) < \infty$.

We define the spectrally negative (i.e. a Lévy process with negative jumps) process $X=(X_t)_{t\geq 0}$ as 
\begin{equation}\label{model1}
 X_t=ct - S_t + Z_t \;,
\end{equation}
where $Z$ is a multiple of a standard Brownian motion, we write
$$Z_t=\sigma B_t,$$
with $B$ a standard Brownian motion independent of $S$. 
Let $\mathcal{F}:=(\mathcal{F}_{t})_{t\geq 0}$ be the filtration obtained by  $\sigma( S_s,B_s, s \geq 0)$.

We consider a very general setup that generalizes the standard Cramér-Lundberg model. The model discussed in this paper is, 
\begin{equation}\label{riskmodel3}
R_t := x - Y_t \;, \qquad t \geq 0 \;,	
\end{equation} 
where $Y=(Y_t)_{ t \geq 0 }$ is a spectrally positive Lévy process (i.e. a Lévy process with positive jumps) defined by  $Y_t=-X_t$, where $X_t$ is given by (\ref{model1}).  As introduced previously, note that the risk process given by
\begin{equation}\label{modelY}
 R_t=x + ct - S_t + Z_t \;,
\end{equation} 
is on the same spirit as the original perturbed model in Dufresne and Gerber (1991).
The constant $x > 0$ represents the initial surplus, while the process $Y$ represents the cash outflow of an insurance company. The subordinator $S$ represents cumulated claims, and this is why we need it to be increasing since the jumps represent claims paid out. The Brownian motion $Z$ accounts for any fluctuations affecting the components of the risk process dynamics, such as claims arrivals, premium income and investment returns;  $c\,t$ represents premium inflow over the interval of time $[0,t]$.

The premium rate $c$ is assumed to satisfy the net profit condition, 
 precisely $\mathbb{E}[S_1]<c$, which requires
\begin{align}\label{netprofit}
 \int_{(0,\infty)}y\nu(dy)<c \;.
\end{align}
The condition in equation (\ref{netprofit}) implies that the process $Y$ has a negative drift in order to avoid the possibility that $R$ becomes negative almost surely. This condition is often expressed in terms of a safety loading. Indeed, it is standard to write the drift component within $Y$ as a loaded premium. For instance, notice that we can recuperate the classical Cramér-Lundberg model if  $\sigma=0$ where $c:=(1+\theta) \mathbb E [S_1]$ and $S$ is a compound Poisson process modeling aggregate claims. The drift $c$, with a positive safety loading $\theta>0$, is the collected premium rate. We do not use the concept of safety loading in this paper in order to simplify the notation but we stress the fact that this concept is implicitly considered within the drift of $Y$ when we impose the condition in equation (\ref{netprofit}). The classical compound Poisson model can be incorporated in this framework by setting $\nu(dy)=\lambda K(dy)$, where $\lambda$ is the Poisson arrival rate and $K$ is a diffuse claim distribution.

We refer to Asmussen (2000) for an account on the classical risk model, and to Furrer and Schmidli (1994), Yang and Zhang (2001), Huzak \emph{et al.} (2004) and Biffis and Morales (2010) for different generalizations and studies of model (\ref{riskmodel3}).

Now, one of the main objects of interest in ruin theory is the \emph{ruin time}, $\tau_x$, representing the 
first passage time of $R_t$ below zero when $R_0=x$, i.e. 
\begin{equation}\label{def:ruin}
\tau_x:=\inf\{ t > 0 \; : \; Y_t > x \}\;, 
\end{equation}
where we set $\tau_x=+\infty$ if $R_t\geq0$ for all $t\geq 0$. 

Associated with the ruin time $\tau_x$, we have at least two other quantities that contain relevant information on the ruin event from a risk management perspective, namely the deficit at ruin $-R_{\tau_x}=Y_{\tau_x}-x$ and the surplus immediately prior ruin $R_{\tau_x-}=x- Y_{\tau_x^-}$.

Gerber and Shiu (1998) studied the ruin event in the compound Poisson case by analyzing the joint law of all these quantities in one single object, the EDPF. In the following, we define the EDPF under the model (\ref{riskmodel3}).
\begin{defn}\label{standardEDPF}
Let $w$ be a non-negative Borel-measurable function on $ \mathbb{R}_+ \times \mathbb{R}_+$ such that $w(.,0)=0$. For $q \geq 0$, the EDPF associated with the risk process (\ref{riskmodel3}) is defined as 
\begin{equation} \label{edpfs1}
\phi(w;x; q)=\mathbb{E}\Big[e^{-q\tau_x}w( x- Y_{\tau_x^-},Y_{\tau_x}-x ) \Big]\,.
\end{equation}
\end{defn}
Note that the condition $w(\cdot,0)=0$ excludes the event $\{Y_{\tau_x}=x\}$. This possibility is known as creeping and we chose not to consider it in our analysis. For simplicity, we assume that the function $w$ assigns a zero penalty when ruin occurs by continuously crossing over zero. Notice that for a model like (\ref{riskmodel3}), this only happens when ruin is caused by the Brownian motion component of the process $Y$. 

Following the same order of ideas, we study the EDPF under the general context which gives relevant informations on and after the ruin event. More precisely, we generalize the EDPF defined in (\ref{edpfs1}) to include the quantities associated with the ruin time $\tau_x$ and  times sequence of successive minima reached by a claim of the risk process (\ref{riskmodel3}) after ruin. This implies that some of notations related to  record minima need to be introduced .


Thus, we must define the first new record time of the running supremum
\begin{align}
 \tau:=\inf \{ t>0, Y_t>\overline{Y}_{t^-}\} \;,
\end{align}
and the sequence of times corresponding to new records of $Y$ reached by a jump of $S$ after the ruin time. More precisely, let
\begin{align}
 \tau^{(1)}:=\tau_x,
\end{align}

and inductively on $\{ \tau^{(n)} < \infty \}$,
\begin{align}
\tau^{(n+1)}:=\inf \{ t>\tau^{(n)}, Y_t>\overline{Y}_{t^-}\}.
\end{align}
Recall from Theorem 4.1 of Huzak \emph{et al.} (2004) that the sequence $(\tau^{(n)})_{n\geq 1}$ is discrete, and, in particular, neither time $0$ non any other time is an accumulation point of those times. More precisely, $\tau > 0$ a.s. and $\tau^{(n)}<\tau^{(n+1)}$ a.s. on $\{ \tau^{(n)} <\infty \}$. As a consequence, we can order the sequence $(\tau^{(n)})_{n\geq 1}$ of times when a new supremum is reached by a jump of a subordinator as $ 0<\tau^{(1)}<\tau^{(2)}<\cdots \quad $ a.s.

Let us introduce the random number $N$ given by 
\begin{equation}\label{N}
N:=\max \{ n: \tau^{(n)} < \infty \},                                                                                                                                                                                                                                                                                                                                                                                                                                                                                                                                                                                                                                                                                                                                             \end{equation}
which represents the number of new records reached by a claim of the risk process (\ref{riskmodel3}). 
In the following, we study the EDPF in a new context involving the deficits at times $(\tau^{(n)})_{n\geq1}$. More precisely, in this paper we set out to study the following extended EDPF for the model (\ref{riskmodel3})
\begin{defn}\label{penality function}
Let $ F=(F_n)_{n \geq 0}$ be a sequence of non-negative measurable functions from $\mathbb{R}_{+} \times \mathbb{R}_{+}$ to $\mathbb{R}$, $x$ and $q \geq 0$. The discounted penalty associated with the risk process (\ref{riskmodel3}), $F$ and $q$ is defined as 
\begin{align}\label{def}
P(F,q,x)=\mathbb{E}\Big[\sum_{n=1}^N e^{-q\tau^{(n)}}F_n( Y_{\tau^{(n-1)}},Y_{\tau^{(n)}} ) ; \tau_x < \infty  \Big].
\end{align}
\end{defn}
We assume in the previous definition that $\tau^{(0)}=\tau_x^-$ and 
\begin{align}\label{cdt1}
F_1(\cdot,x)=0.
\end{align}
Note that the condition given by (\ref{cdt1}) is used to exclude from calculation the event $\{Y_{\tau_x}=x\}$. $P(F,q,x)$ is an extension of the classical EDPF defined in (\ref{edpfs1}). In particular, it reduces to $\phi(w;x; q)$ since we suppose that  $F_1(u,v)=w(x-u, v-x)$ and $F_n=0$ for $n\geq 2$.

\section{Preliminary results}\label{preliminary result}
In this section, we will give some preliminary results for first-passage times of the risk process defined in (\ref{riskmodel3}) under change of measure. These results are based on the works in Huzak \emph{et al.} (2004) where the ruin probability has been studied for a subordinator risk model perturbed by a spectrally negative Lévy process. This allows us to give a more detailed analysis of the extended EDPF defined in (\ref{def}).

Recall from Section \ref{riskEDPF} that $S$ is a subordinator defined on the filtred probability space $(\Omega, \mathbf{F}, \mathbb{P})$. The Laplace exponent of $S$ is defined by
\begin{equation}\label{laplaceS}
 \psi_S(\alpha)= \int_{(0,\infty)}[e^{\alpha y} -1]\nu(dy)\; , 
\end{equation}
where
\begin{equation}\label{explaplaceS}
 \mathbb{E}[\exp(\alpha S_t)] = \exp(t\psi_S(\alpha))\; . 
\end{equation}
Note that 
\begin{equation}
\mathbb{E}[ S_1] = \psi_S^,(0^+)=\int_{(0,\infty)}y\nu(dy)=\int_0^{\infty}\nu(y,\infty)dy\; , 
\end{equation}
where the last equality follows from integration by parts. As explained before, we assume throughout that $\mathbb{E}[ S_1] < \infty$.

The Laplace exponent $\psi$ of $X$ defined in (\ref{model1}) is defined by the relation
\begin{eqnarray}
\mathbb{E}[\exp(\beta X_t)]&=&\exp(t \psi_X(\beta)),\nonumber\\
\end{eqnarray}
where
\begin{eqnarray}
 \psi_X(\beta)&=&c\beta + \psi_S(-\beta) + \psi_Z(\beta)\nonumber\\
&=&\psi_S(-\beta)+\widetilde{Z}(\beta)\quad\quad \beta \geq 0\; ,\nonumber
\end{eqnarray}
where $\widetilde{Z}_t=ct + Z$ and $\psi_{\widetilde{Z}}(\beta)=c\beta + \psi_Z(\beta)$. The last equality is due to the independence of $S$ and $Z$. We refer to Bertoin (1996), Sato (1999) and Kyprianou (2006) for a comprehensive account on Lévy process.

Let us introduce the distribution function $G$ of $-\inf_{t\geq 0}( \widetilde{Z}_t)=\sup_{t\geq 0}(-ct - Z_t)$. Using a method similar to Yang and Zhang (2001) (see also Huzak \emph{et al.} (2004)), the Laplace transform of $G$ can be shown to be given by
\begin{eqnarray}
\widehat{G}(\beta)&=&\int_0^{\infty}e^{-\beta y}G(dy)\nonumber\\
&=&\frac{c\beta}{\psi_Z(\beta)}\;.
\end{eqnarray}
That is, 
\begin{eqnarray}\label{Gdensity}
\widehat{G}(\beta)&=&\frac{c\beta}{c\beta + \frac{\sigma^2 \beta^2}{2}}\;,
\end{eqnarray}
since $Z_t=\sigma B_t$ is a multiple of a standard Brownian motion. Then, $G$ is given explicitly as an exponential distribution function with parameter $2c/\sigma^2$, i.e. $G$ has density 
\begin{equation}\label{density1}
G(dy)= \frac{2c}{\sigma^2}e^{-\frac{2c}{\sigma^2}y}dy\;.
\end{equation}
We also introduce the parameter
\begin{eqnarray}
 \rho&:=&\frac{\mathbb{E}[S_1]}{c}= \frac{1}{c}\int_{(0,\infty)}y\nu(dy)\in (0,1)\;.  
\end{eqnarray}
We denote by $\overline{Y}$ the supremum given by  $\overline{Y}_t = \sup_{s \geq t}Y_s$, for $t\geq0$.

Let us introduce the density of the overshoot at time $\tau$
which is defined by
\begin{equation}\label{density2}
 H(du)=\mathbb{P}(Y_{\tau}-\overline{Y}_{\tau^-}\in du; \; \tau < \infty) \;,
\end{equation}
where $u>0$.
We shall give  in the following proposition the density $H(\cdot)$ in terms of Lévy measure and the premium rate.
\begin{prop}\label{propo1}
Let $Y$ the spectrally-positive Lévy process defined in (\ref{modelY}).\\
1-The distribution of $Y_{\tau}-\overline{Y}_{\tau^-}$ on the set $\tau < \infty$ is given by 
\begin{equation}
 H(du)=\frac{1}{c}\int_0^{\infty}\nu(du+y)dy\;; \;\; u>0\;.
\end{equation}
2- The distribution of $Y_{\tau}$ is given by
\begin{equation}
 \mathbb{P}(Y_{\tau} \in du;\; \tau < \infty)=H \ast G(du) \;; \;\; u>0 \;,
\end{equation}
where $H \ast G(\cdot)$  denotes the convolution of $H(\cdot)$ with $G(\cdot)$ defined by $$ \int_{A}f(u)H \ast G(du)=\int_{\{y+v\in A\}}f(y+v)H(dy)G(dv) \;,$$
for all Borel set $A$ of $\mathbb{R}\times \mathbb{R}$.
\end{prop}
\begin{proof}
1- Let us suppose that $f$ is a nonegative bounded Borel function. We prove firstly that 
\begin{eqnarray}\label{identity1}
\mathbb{E}[f(Y_{\tau}-\overline{Y}_{\tau^-}) ;\; \tau < \infty]&=&\mathbb{E}\big[\int_0^{\tau}\widetilde{f}(\overline{Y}_t-Y_t) dt\big] \;,
\end{eqnarray}
where $\widetilde{f}(y)=\int_{(0,\infty)}f(u-y) 1_{\{ u > y\}}\nu(du)$. The proof of (\ref{identity1}) follows by an application of the compensation formula [see Bertoin(1996) p.9 or Theorem 4.4. of Kyprianou (2006)] applied to the Poisson random measure, with intensity measure $dt \;,\nu(du)$ associated to the jump of $Y$. We have 
\begin{eqnarray}\label{identity2}
\mathbb{E}[f(Y_{\tau}-\overline{Y}_{\tau^-}) ;\; \tau < \infty]&=&\mathbb{E}[\sum_{t> 0}f(Y_t-\overline{Y}_{t-}) 1_{\{ Y_t > \overline{Y}_{t-}, t \leq \tau\}}]\nonumber \\
&=&\mathbb{E}[\int_0^{\infty}\int_{[0,\infty)}f(u-\overline{Y}_t + Y_t) 1_{\{ u > \overline{Y}_t-Y_t, t \leq \tau\}}]dt\nu(du)\nonumber \\
&=&\mathbb{E}[\int_0^{\tau}\int_{(0,\infty)}f(u-(\overline{Y}_t-Y_t)) 1_{\{ u > \overline{Y}_t-Y_t \}}]dt\nu(du)\nonumber\\
&=&\mathbb{E}\big[\int_0^{\tau}\widetilde{f}(\overline{Y}_t-Y_t) dt\big].
\end{eqnarray}
Hence, from Proposition 4.3. of Huzak \emph{et al.} (2004), the expected occupation time measure of (\ref{identity2}) is given by
\begin{eqnarray}\label{identity3}
\mathbb{E}\big[\int_0^{\tau}\widetilde{f}(\overline{Y}_t-Y_t) dt\big]&=&\frac{\mathbb{P}(\tau=\infty)}{c-\mathbb{E}[S_1]}\int_0^{\infty}\widetilde{f}(y)dy.
\end{eqnarray}
Since $\mathbb{P}(\tau=\infty)=1-\mathbb{P}(\tau<\infty)=1-\rho$ [see Corollary 4.5. of Huzak \emph{et al.} (2004)], (\ref{identity3}) is equal to
\begin{eqnarray}\label{identity4}
\frac{1-\rho}{c-\mathbb{E}[S_1]}\int_0^{\infty}\widetilde{f}(u)du &=&\frac{1}{c}\int_0^{\infty}\int_{(0, \infty)}f(u-y) 1_{\{ u > y\}}\nu(du)dy \nonumber \\
&=&\frac{1}{c}\int_{(0, \infty)}f(u)\int_0^{\infty}\nu(du+y)dy.
\end{eqnarray}
Equating the left-hand side of (\ref{identity1}) and (\ref{identity4}) implies that 
\begin{eqnarray}\label{identity5}
\mathbb{P}(Y_{\tau}-\overline{Y}_{\tau^-}\in du; \; \tau < \infty )&=&\frac{1}{c}\int_{(0, \infty)}\int_0^{\infty}\nu(du+y)dy\;,
\end{eqnarray}
and the statement 1 of Proposition \ref{propo1} follows.

2- Using the conditional independence of $Y_{\tau}-\overline{Y}_{\tau^-}$ and $\overline{Y}_{\tau^-}$ given $\tau<\infty$,
\begin{eqnarray}
\mathbb{P}(Y_{\tau}\in du\;; \; \tau < \infty )&=&\mathbb{P}(Y_{\tau}\in du | \; \tau < \infty )\mathbb{P}( \tau < \infty )\\
&=&\mathbb{P}(Y_{\tau}-\overline{Y}_{\tau^-} +\overline{Y}_{\tau^-}\in du | \; \tau < \infty )\rho \\
&=& H\ast \overline{G}(du)\;,
\end{eqnarray}
where  $\overline{G}(du)$ is the conditional distribution $\mathbb{P}(\overline{Y}_{\tau^-}\in du | \; \tau < \infty )$.
From corollary 4.6 of Huzak \emph{et al.} (2004), $\overline{G}(du)$ is equal to unconditionnal distribution $\mathbb{P}(\overline{Y}_{\tau^-}\in du)$. Corollary 4.10 of Huzak \emph{et al.} (2004) implies that
 \begin{eqnarray}\label{identity6}
\mathbb{P}(\overline{Y}_{\tau^-}\in du  )&=&\mathbb{P}(\sup_{t\geq 0}(-ct - \sigma B_t)\in du ),
\end{eqnarray}
from which it immediately follows that $\overline{G}=G$. Thus we have proved the the statement 2 of proposition.
\end{proof}

For the right inverse of $\psi$,  we shall write $\phi$ on $[0,\infty)$. That is to say, for each $q\geq 0$,
\begin{align}\label{inverse}
 \Phi(q)=\sup\{ \alpha \geq 0: \psi(\alpha)=q\}.
\end{align}
Note that the properties of $\psi$ for such a Lévy process $X$, imply that $\phi(q)> 0$ for $q > 0$. Further $\phi(0)=0$,  since $\psi'(0)=c-\mathbb{E}[S_1] \geq 0$.

Let us define the probability measure $ \widetilde{\mathbb{P}}$  by the density process
\begin{align}\label{change measure}
\frac{d\widetilde{\mathbb{P}}}{d\mathbb{P}}\Big|_{\mathcal{F}_t}=e^{-\Phi(q) Y_t-qt}\, 
\end{align}
where $ \Phi(q)$ is the right inverse of $\psi$ defined in (\ref{inverse}). Note that under $\widetilde{\mathbb{P}}$ the process $Y$ introduced in (\ref{model1}), is still a spectrally positive Lévy process, and still drifts to $-\infty$ [see Kyprianou (2006)]. In addition, the process $Y$ keeps the same form under $\widetilde{\mathbb{P}}$ and then, $Y_t=-\widetilde{c}t + S_t-\sigma \widetilde{B}_t$, where $\widetilde{c}=c+\sigma^2\Phi(q)$  and $\widetilde{B}_t= B_t-\sigma \Phi(q) t$ are respectively the premium rate and the standard Brownian motion under the probability measure $ \widetilde{\mathbb{P}}$.

We denote by $\widetilde{\nu} $ the Lévy measure of $S$ under the change of measure $ \widetilde{\mathbb{P}}$ and then,
\begin{equation}
 \widetilde{\nu}(du)=e^{-\Phi(q)u}\nu(du).
\end{equation}
 
Let
\begin{align}\label{rho2}
\widetilde{\rho}=\frac{\widetilde{\mathbb{E}}[S_1]}{\widetilde{c}}
=\frac{\int_{(0,\infty)}ye^{-\Phi(q)y}\nu(dy)}{c+\sigma^2\Phi(q)^2} \;,
\end{align}
then by (\ref{netprofit}), $0<\widetilde{\rho}<1$ and the net profit condition is well preserved under $\widetilde{\mathbb{P}}$.
Note that the distribution function of $-\inf_{t\geq 0}( Z_t)=\sup_{t\geq 0}(-ct - \sigma B_t)$ under the change of measure, is denoted by $\widetilde{G}$. Recall from  
(\ref{Gdensity}) that the density $\widetilde{G}$ is given via its Laplace transform by
\begin{eqnarray}\label{Gdensity1}
\int_0^{\infty}e^{-\beta y}\widetilde{G}(dy)
&=&\frac{\widetilde{c}\beta}{\widetilde{c}\beta + \frac{\sigma^2 \beta^2}{2}},\quad \beta >0;
\end{eqnarray}
and then, $\widetilde{G}$ is given explicitly as an exponential distribution function with parameter $2\widetilde{c}/\sigma^2$. We denote by $\widetilde{H}$ the distribution of  $Y_{\tau}-\overline{Y}_{\tau^-}$ under the probability measure $ \widetilde{\mathbb{P}}$, i.e 
\begin{eqnarray}
 \widetilde{H}(du)=\widetilde{\mathbb{P}}(Y_{\tau}-\overline{Y}_{\tau^-}\in \; du\;;\; \tau < \infty), \text{  for } u>0.
\end{eqnarray}
Moroever, since the characteristics of the risk process are preserved under the change of measure, we can derive a result analogous to Proposition \ref{propo1} under $\widetilde{\mathbb{P}}$ as:
\begin{prop}
Let $Y$ be a spectrally positive Lévy process defined in (\ref{modelY}).\\
1-The distribution of $Y_{\tau}-\overline{Y}_{\tau^-}$ on the set $\tau < \infty$ under the probability measure $\widetilde{\mathbb{P}}$ is given by 
\begin{eqnarray}
 \widetilde{H}(du)&=&\frac{1}{\widetilde{c}}\int_0^{\infty}\widetilde{\nu}(du+y)dy\nonumber\\
&=&\frac{e^{-\Phi(q)u}}{c+\Phi(q)\sigma^2}\int_0^{\infty}e^{-\Phi(q)y}\nu(du+y)dy\;; \; u>0.
\end{eqnarray}
2- The distribution of $Y_{\tau}$ under the probability measure $\widetilde{\mathbb{P}}$ is given by
\begin{equation}\label{idchap4}
 \widetilde{\mathbb{P}}(Y_{\tau} \in du\;;\; \tau < \infty)=\widetilde{H} \ast \widetilde{G}(du) \;; \; u>0.
\end{equation}
\end{prop}

In the following paragraph, we shall give the ruin probability under the change of measure defined by (\ref{change measure}).

 We denote by  $\widetilde{\theta}(x)$ the ruin probability under the probability measure $ \widetilde{\mathbb{P}}$, that is  $$\widetilde{\theta}(x)=\widetilde{\mathbb{P}}(\sup_{t\geq 0}Y_t < x)\;, \quad \quad x\geq 0. $$ 
For $q\geq 0$, the following proposition gives the Pollazek-Hinchin formula for the survival probability under the change measure $\widetilde{\mathbb{P}}$.

\begin{prop}
The survival probability of the general perturbed risk process introduced in  (\ref{riskmodel3}) is given by 
 \begin{equation}
1-\widetilde{\theta}(x)=(1-\widetilde{\rho})\sum_{n=0}^{\infty}\big(\widetilde{L}^{\ast(n)}\ast \widetilde{G}^{\ast(n+1)}\big)(x)\widetilde{\rho}^n\;;
 \end{equation}
 where $$\widetilde{L}(dy)=\frac{1}{\widetilde{c}}\widetilde{\nu}(y, \infty)dy 1_{\{y>0\}}\;, $$
and $\widetilde{G}(\cdot)$ is an exponential distribution function with parameter $2\widetilde{c}/\sigma^2$, $f^{\ast n}$ $(n\geq 1)$ denotes the $n$-fold convolution of $f$ with itself and $f^{\ast 0}$ is the distribution function corresponding to the Dirac measure at zero.

\end{prop}

\begin{proof}
Using a similar method to that in Huzak \emph{et al.} (2004), we have by taking limits in the Laplace transform of the infimum evaluated at an independant exponential time $e_q$ with parameter $q> 0$ [see chapter VIII in Kyprianou (2006)]\;,
\begin{eqnarray} \label{Laplace1}
\widetilde{\mathbb{E}}[e^{\beta \underline{X}_{\infty}}]&=&\widetilde{\mathbb{E}}[e^{-\beta \overline{Y}_{\infty}}]\nonumber \\&=&\widetilde{\psi}'(0+)\frac{\beta}{\widetilde{\psi}(\beta)},\text{  for }\beta > 0.
\end{eqnarray}

Let us now compute (\ref{Laplace1}) in terms of $\widetilde{\rho}$, $\widehat{\widetilde{L}}$ and $\widehat{\widetilde{G}}$, where $\widetilde{\rho}$ is the parameter given by (\ref{rho2}), and $\widehat{\widetilde{L}}$ and $\widehat{\widetilde{G}}$ are respectively the Laplace transforms of $\widetilde{L}$ and $\widetilde{G}$. Equation (\ref{Laplace1}) is equal to
\begin{eqnarray}\label{Laplace2}
\frac{\widetilde{d}}{\widetilde{c}}\quad \frac{\widehat{\widetilde{G}}(\beta)}{1-\widetilde{\rho}\widehat{\widetilde{G}}(\beta)\widehat{\widetilde{L}}(\beta)}
&=&(1-\widetilde{\rho})\widehat{\widetilde{G}}(\beta)\sum_{n=0}^{\infty}\big(\widetilde{\rho}\widehat{\widetilde{G}}(\beta)\widehat{\widetilde{L}}(\beta)\big)^n.
\end{eqnarray}

By inverting the Laplace transform  (\ref{Laplace2}), we obtain
\begin{equation}
\widetilde{\mathbb{P}}(\overline{Y}_{\infty}\leq x)=1-\widetilde{\theta}(x)=(1-\widetilde{\rho})\sum_{n=0}^{\infty} \widetilde{\rho}^n \big(\widetilde{G}^{\ast (n+1)} \ast \widetilde{L}^{\ast n}\big)(x).
\end{equation}
\end{proof}

We now introduce the so-called \emph{$q$-scale function} $\{W^{(q)},\;q\geq 0\}$ of the process $X$. For every $q\geq 0$,
there exists a function $W^{(q)}: \mathbb{R}\longrightarrow [0, \infty)$ such that $W^{(q)}(y)=0$ for all $y < 0$ and otherwise absolutely continuous on $(0,\infty)$ satisfying
\begin{equation}\label{scale1}
 \int_0^{\infty}e^{-\lambda y}W^{(q)}(y)dy=\frac{1}{\psi(\lambda)-q}\,, \; \text{ for}\quad \lambda >\Phi(q)\;,
\end{equation}
where $\Phi(q)$ is the largest solution of $\psi(\beta)=q$ defined in (\ref{inverse}).

For short, we shall write $W^{(0)}=W$. Let us introduce the $0$-scale function under $ \widetilde{\mathbb{P}}$, which we write as $W_{\Phi(q)}$, related to the $q$-scale function of $X$,  $W^{(q)}$, via the relation 
\begin{equation}\label{scale2}
 W^{(q)}(y)=e^{\Phi(q)y}W_{\Phi(q)}(y).
\end{equation}
The reader is otherwise referred to Bertoin (1996) and Chapter 8 of Kyprianou (2006) for a fuller account.

In the following, we shall describe the discounted penalty function introduced in Definition \ref{penality function} in terms of densities $\widetilde{G}$, $\widetilde{H}$ and $q$-scale function of the spectrally negative process $X$ which we have briefly introduced above.

\section{Extension of the Expected Discounted Penalty Function }
At this point, we recall that the main objective of this paper is to write an expression for the extended EDPF in (\ref{def}). But before we can write out such an expression we need one more intermediate result that has to do with the change of measure defined through the density process in (\ref{change measure}). 

Let us start by giving the standard EDPF defined in (\ref{edpfs1}) in terms of the convolution product of two functions depending on Lévy measure $\nu$ and $q$-scale function which is defined by (\ref{scale1}). Recall that $w$ is a non-negative  bounded measurable function on $\mathbb{R}\times  \mathbb{R}$ such that $w(\cdot,0)=0$.

\begin{lem}\label{lem1}
Consider the risk model in (\ref{riskmodel3}).
\begin{enumerate}
\item The distribution of the overshoot at $\tau_x$, $\widetilde{T}_x$, is given by
\begin{equation}\label{overshoot}
\widetilde{T}_x(du)=\int_0^{\infty}\int_0^ve^{-\Phi(q)u}\nu(du-x+v) W'_{\Phi(q)}(x-y)dy dv\;,
\end{equation}
where $\widetilde{T}_x(du)=\widetilde{\mathbb{P}}( Y_{\tau_x} \in du ; \tau_x < \infty  )$. 
\item For $q\geq 0$, the EDPF $\phi(w,q,x)$, as introduced in Definition \ref{standardEDPF}, is given by,
\begin{equation}\label{G-S}
 f_1 \ast f_2(x)\;,
\end{equation} 
where
\begin{eqnarray}
 f_1(x)&=&e^{\Phi(q)x}W'_{\Phi(q)}(x)\nonumber \\
&=& W'^{(q)}(x)-\Phi(q)W^{(q)}(x)\;, 
\end{eqnarray}
and
\begin{eqnarray}
 f_2(x)&=&e^{\Phi(q)x}\int_x^{\infty}e^{-\Phi(q)v}\int_{(0, \infty)}w(u,v)\nu(du+v) dv\;,
\end{eqnarray}
for $u>x,v>0$ and $0<y<x\wedge v$.
\end{enumerate}
\end{lem}

\begin{proof}
1). Follow the method used in Biffis and Kyprianou (2010) [see also the end of Section 8.4 of Kyprianou (2006)] and recalling that $X$ drifts to $\infty$ (and hence $\Phi(0)=0$), we know that 
\begin{eqnarray}\label{f1}
\mathbb{P}(  Y_{\tau_x}-x \in du,x -Y_{\tau_x^-} \in dv ; \tau_x < \infty )
&=& \nu(du+v)[W(x)-W(x-v)]dv\nonumber\\
&=&\nu(du+v)\int_0^v W'(x-y)dydv \;,\nonumber\\
\end{eqnarray}

for $u>0,\;v>0$ and $0<y<x\wedge v$, where $W'$ is a version of the density of $W$.

Using the previous equality under the change of measure $\widetilde{\mathbb{P}}$, we obtain the identity
\begin{eqnarray}\label{ff2}
\widetilde{\mathbb{P}}( Y_{\tau_x}-x \in du, x-Y_{\tau_x^-} \in dv \;; \; \tau_x < \infty )
&=&\widetilde{\nu}(du+v)\int_0^v W'_{\Phi(q)}(x-y)dydv\nonumber\\
&=& e^{-\Phi(q)(u+v)}\nu(du+v)\nonumber \\&& \times \int_0^v W'_{\Phi(q)}(x-y)dydv\;,\nonumber \\
\end{eqnarray}
for $u>0,v>0$, $0<y<x\wedge v$ and then
\begin{eqnarray}
\mathbb{E}\Big[ e^{-q\tau_x}w( Y_{\tau_x}-x,x-Y_{\tau_x^-} ) ; \tau_x < \infty  \Big]&=&\widetilde{\mathbb{E}}\Big[ e^{\Phi(q)Y_{\tau_x}}w( Y_{\tau_x}-x,x-Y_{\tau_x^-} ) ; \tau_x < \infty  \Big]\nonumber\\
&=&\int_{(0,\infty)}\int_{(0,\infty)}e^{\Phi(q)(u+x)}w(u,v)\nonumber\\&& \times \widetilde{\mathbb{P}}(  Y_{\tau_x}-x \in du, x-Y_{\tau_x^-} \in dv ; \tau_x < \infty )\nonumber \\
&=&\int_0^xe^{\Phi(q)(x-y)}W'_{\Phi(q)}(x-y)e^{\Phi(q)y}\nonumber\\&& \times\int_y^{\infty}e^{-\Phi(q)v}\int_{(0,\infty)}w(u,v)\nu(du+v)dvdy\nonumber\\
&=&f_1 \ast f_2(x)\;,
\end{eqnarray}
where
\begin{equation}\label{f2}
  f_2(x)=e^{\Phi(q)x}\int_x^{\infty}e^{-\Phi(q)v}\int_{(0, \infty)}w(u,v)\nu(du+v) dv.      
\end{equation}
By using  (\ref{scale2}) we get,
\begin{eqnarray}\label{f1}
f_1(x)&=&e^{\Phi(q)x}W'_{\Phi(q)}(x)\nonumber \\
&=& W'^{(q)}(x)-\Phi(q)W^{(q)}(x)\;, 
\end{eqnarray}
from which, statement in 1)  holds.

2). From (\ref{ff2}), we deduce (\ref{overshoot}) by writing,
\begin{eqnarray}
\widetilde{T}_x(du)&=&\widetilde{\mathbb{P}}(Y_{\tau_x}-x \in du-x ; \tau_x < \infty)\nonumber\\
&=&\int_0^{\infty}\int_0^ve^{-\Phi(q)u}\nu(du-x+v) W'_{\Phi(q)}(x-y)dy dv,
\end{eqnarray}
for $u>x,v>0$ and $0<y<x\wedge v$.
\end{proof}
\begin{rem}\label{rem1}
Equation (\ref{G-S}) is equivalent to the following equality  which describes the Gerber-Shiu function in terms of scale function of risk process;
\begin{eqnarray}\label{Edpf}
\phi(w,q,x)&=& \int_{(0,\infty)}\int_{(0,\infty)}e^{-\Phi(q)(x-v)}w(u,v)\nu(du+v)\big[W^{\Phi(q)}(x)-W^{\Phi(q)}(x-v)\big]dv\nonumber\\
&=&\int_0^{\infty }\int_0^{\infty }w(v,u)\big[e^{-\phi(q)v}W^{(q)}(x)-W^{(q)}(x-v)\big]\nu(du+v)dv\;,
\end{eqnarray}
by using (\ref{scale2}), where $w$ is a bounded measurable function such that $w(\cdot,0)=0$.
\end{rem}

Identity (\ref{Edpf}) is given in more general form in Biffis and Kyprianou (2010), where the EDPF also includes the size of the last minimum before ruin $x-\overline{Y}_{\tau_x-}$.

Recall that from Section \ref{preliminary result} that $N$ is the random number given by (\ref{N}). We give in the following proposition the distribution of $N$ under the probability measure defined by (\ref{change measure}). 
\begin{prop}\label{loiN}
The distribution of $N$ on $\{ \tau_x < \infty \}$ is given by
\begin{align}
 \widetilde{\mathbb{P}}(N=n,\tau_x < \infty)=(1-\widetilde{\rho})\widetilde{\rho}^n(1-\widetilde{\theta}(x))\;;
\end{align}
where $n\geq 0$.
\end{prop}
\begin{proof}
By the strong Markov property of $\overline{Y}$ under $\widetilde{\mathbb{P}}$, we can identify the distribution of $N$ on $\{\tau_x < \infty \}$ as 
\begin{eqnarray}
\widetilde{\mathbb{P}}(N=0,\tau_x < \infty)&=&\widetilde{\mathbb{P}}(\tau^{(1)}=\infty |\tau_x < \infty)\widetilde{\mathbb{P}}(\tau_x < \infty)\nonumber \\&=&\widetilde{\mathbb{P}}(\tau = \infty )(1-\widetilde{\theta}(x)) \nonumber \\
&=& (1-\widetilde{\rho})(1-\widetilde{\theta}(x)) \;,
\end{eqnarray}
where in the last equality we have used Corollary 4.5 of Huzak \emph{et al.} (2004) under the probability measure $\widetilde{\mathbb{P}}$.

For $ n \geq 1$,
\begin{eqnarray}
\widetilde{\mathbb{P}}(N=n,\tau_x < \infty)
&=&\widetilde{\mathbb{P}}(\tau^{(n+1)}=\infty |\tau^{(n)} < \infty)\widetilde{\mathbb{P}}(\tau^{(n)} < \infty|\tau^{(n-1)} < \infty) \times\nonumber \\
&&  ...\times \widetilde{\mathbb{P}}(\tau^{(1)}<\infty |\tau_x < \infty)\widetilde{\mathbb{P}}(\tau_x < \infty)\nonumber \\
&=&\widetilde{\mathbb{P}}(\tau = \infty )\widetilde{\mathbb{P}}(\tau < \infty )^n(1-\widetilde{\theta}(x))\nonumber \\&=&
(1-\widetilde{\rho})\widetilde{\rho}^n(1-\widetilde{\theta}(x)).
\end{eqnarray}

\end{proof}

Now, we are finally in a position to present the first and main contribution of this paper. In the following theorem, we give an expression for the extended EDPF $P(F;q;x)$ given in Definition \ref{penality function}. This expression is given in terms of the $q$-scale function, Lévy measure, and the densities $\widetilde{G}$ and $\widetilde{H}$ which are introduced in Section \ref{preliminary result}. Recall that $ F=(F_n)_{n \geq 0}$ is a sequence of non-negative measurable functions from $\mathbb{R}_{+} \times \mathbb{R}_{+}$ to $\mathbb{R}$ such that $F_0(.,x)=0$.
\begin{thm}\label{characterization}
Consider the risk model in (\ref{riskmodel3}). For $q\geq 0$, the EDPF $P(F,q,x)$, as introduced in Definition \ref{penality function}, is given by,
\begin{align}\label{SEDPF}
 \phi(w,q,x)+\sum_{n=1}^{\infty} \int_{(x,\infty)}\int_{(0,\infty)}e^{\Phi(q)(u+v)} F_{n+1}( v,u+v ) \nonumber \\ \widetilde{H}\ast \widetilde{G}(du)\widetilde{H}^{\ast n}\ast \widetilde{G}^{\ast n}\ast \widetilde{T}_x(dv)\;,
 \end{align}
where $w$ is a mesurable Borel-function satisfying  $w(u,v)=F_1( x-v,u-x)$ for $u ,\; v \geq 0 $ and 
\begin{align}
 \phi(w,q,x)=\mathbb{E}\Big[ e^{-q\tau_x}w( x-Y_{\tau_x^-},Y_{\tau_x}-x ) \;; \tau_x < \infty  \Big].
\end{align}

\end{thm}

\begin{proof}
We prove the result in three steps. Let us suppose  $F=(F_n)_{n\geq 1}$  is a sequence of non-negative measurable functions from $\mathbb{R}_{+} \times \mathbb{R}_{+}$ to $\mathbb{R}$, $x$ and $q \geq 0$.

Step 1: We prove in this step that 
\begin{align}\label{step1}
\widetilde{\mathbb{P}}(Y_{\tau^{(k)}}-Y_{\tau^{(k-1)}} \in  dy |  \tau^{(k)}<\infty)=\frac{1}{\widetilde{\rho}}\;\widetilde{H }\ast \widetilde{G}(dy);
\end{align}
where $k \geq 2$ and $y \geq 0$. The proof of (\ref{step1}) follows by an application of Theorem 4.7, Huzak et al (2004). Recall that under the measure change $\widetilde{\mathbb{P}}$,  $Y$ has the same form and it can be written as $Y_t=-\widetilde{c}t + S_t - \sigma \widetilde{B}_t$. By using the Markov proprety of $\overline{Y}$ at $\tau^{(k)}$  
\begin{eqnarray}
\widetilde{\mathbb{P}}(Y_{\tau^{(k)}}-Y_{\tau^{(k-1)}} \in dy |  \tau^{(k)}<\infty)
&=&\widetilde{\mathbb{P}}(Y_{\tau}\in dy |  \tau < \infty) \nonumber \\
&=& \frac{1}{\widetilde{\rho}}\;\widetilde{H }\ast \widetilde{G}(dy),
\end{eqnarray}
where in the last equality we have used Equation (\ref{idchap4}) and the identity  $\widetilde{\mathbb{P}}(\tau < \infty)=\widetilde{\rho}$ [see Corollary 4.6, Huzak \emph{et al.} (2004)].

Step 2: Next we prove that 
\begin{align}\label{step2}
\widetilde{\mathbb{P}}(Y_{\tau^{(k)}}\in dy |  \tau^{(k)}<\infty)=\frac{1}{\widetilde{\rho}^k(1-\widetilde{\theta}(x))}\;(\widetilde{H}^{\ast k}\ast \widetilde{G}^{\ast k}\ast \widetilde{T}_x)(dy),
\end{align}
where $k \geq 2$ and $y \geq 0$. 

\begin{eqnarray}
\widetilde{\mathbb{P}}(Y_{\tau^{(k)}}\in dy |  \tau^{(k)}<\infty)
&=&\widetilde{\mathbb{P}}( Y_{\tau^{(k)}}-Y_{\tau^{(k-1)}}+ ...+ Y_{\tau^(1)}-Y_{\tau_x}+ Y_{\tau_x} \in dy |  \tau^{(k)}<\infty)\nonumber \\
&=& \frac{1}{\widetilde{\rho}^k}\;(\widetilde{H}^{\ast(k)}\ast \widetilde{G}^{\ast(k)}\ast \widetilde{T}_x)(dy) \frac{1}{1-\widetilde{\theta}(x)},
\end{eqnarray}
by using the independent increments of $Y$ and (\ref{step1}), where $\widetilde{T}_x(dy)=\widetilde{\mathbb{P}}(Y_{\tau_x}\in dy\;;\tau_x < \infty)$.

Step 3: Before using the conclusions of step 1 and step 2, let us write $P(F,q,x)$ as an expansion under  $\widetilde{\mathbb{P}}$,
\begin{eqnarray}\label{step3}
P(F,q,x)&=&\sum_{n=0}^{\infty} \mathbb{E}\Big[\sum_{k=0}^n e^{-q\tau^{(k)}}F_{k+1}( Y_{\tau^{(k)}},Y_{\tau^{(k+1)}} )\;; N=n; \tau_x < \infty \Big] \nonumber \\
&=&\mathbb{E}\Big[e^{-q\tau_x}F_1( Y_{\tau_x^-},Y_{\tau_x}\;; N=0;\tau_x < \infty )\Big]\nonumber \\
&& +\sum_{n=1}^{\infty} \mathbb{E}\Big[\sum_{k=0}^n e^{-q\tau^{(k+1)}}F_{k+1}( Y_{\tau^{(k)}},Y_{\tau^{(k+1)}} )\; ; N=n;\tau_x < \infty \Big] \nonumber \\
&=& \mathbb{E}\Big[e^{-q\tau_x}F_1( Y_{\tau_x^-},Y_{\tau_x})\; ;N=0;\tau_x < \infty \Big] \nonumber \\
&& +\sum_{n=1}^{\infty} \Big(\mathbb{E}\Big[e^{-q\tau_x}F_1( Y_{\tau_x^-}, Y_{\tau_x})\; ; N=n;\tau_x < \infty \Big] \nonumber \\ 
&&+\sum_{k=1}^n \mathbb{E}\Big[e^{-q\tau^{(k+1)}}F_{k+1}( Y_{\tau^{(k)}},Y_{\tau^{(k+1)}} )\; ;N=n\;,\tau_x < \infty \Big]\Big),
\end{eqnarray}
where in the last equality we have used Proposition \ref{loiN}. Then $P(F,q,x)$ is equal to
\begin{eqnarray}
&&\widetilde{\mathbb{E}}\Big[e^{\Phi(q)Y_{\tau_x}}F_1( Y_{\tau_x^-},Y_{\tau_x})\Big|\tau_x < \infty \Big](1-\widetilde{\rho})(1-\widetilde{\theta}(x))\nonumber \\
&&+\sum_{n=1}^{\infty} \Big(\mathbb{E}\Big[e^{-q\tau_x}F_1( Y_{\tau_x^-},Y_{\tau_x})\big|\tau_x < \infty )\Big] +\sum_{k=1}^n \widetilde{\mathbb{E}}\Big[F_{k+1}( Y_{\tau^{(k)}},Y_{\tau^{(k+1)}} )|\tau^{(k+1)} <\infty \Big]\Big)\nonumber \\
&&(1-\widetilde{\rho})\widetilde{\rho}^n(1-\widetilde{\theta}(x))
\end{eqnarray}

\begin{eqnarray}
&=& \widetilde{\mathbb{E}}\Big[F_1( Y_{\tau_x^-},Y_{\tau_x})\Big|\tau_x < \infty )\Big](1-\widetilde{\theta}(x))(1-\widetilde{\rho})[1+\sum_{n=1}^{\infty}\widetilde{\rho}^n ] \nonumber \\
&&+\sum_{n=1}^{\infty} \Big(\sum_{k=1}^n \widetilde{\mathbb{E}}\Big[e^{\Phi(q)Y_{\tau^{(k+1)}}} F_{k+1}( Y_{\tau^{(k)}},Y_{\tau^{(k+1)}} )\Big|\tau^{(k+1)} <\infty \Big]\Big)(1-\widetilde{\rho})\widetilde{\rho}^n(1-\widetilde{\theta}(x))\nonumber \\
&=&\phi(q,w,x) +\sum_{n=1}^{\infty} \Big[\sum_{k=1}^n \int_{(0,\infty)}\int_{(0,\infty)}e^{\Phi(q)(u+v)} F_{k+1}(v,u+v ) \nonumber \\ &&\widetilde{\mathbb{P}}(Y_{\tau^{(k+1)}}-Y_{\tau^{(k)}}\in du,Y_{\tau^{(k)}}\in dv|\tau^{(k+1)} <\infty )\Big](1-\widetilde{\rho})\widetilde{\rho}^n(1-\widetilde{\theta}(x))\\
\end{eqnarray}

\begin{eqnarray}
&=&\phi(q,w,x) +\sum_{n=1}^{\infty}  \int_{(0,\infty)}\int_{(0,\infty)}e^{\Phi(q)(u+v)} F_{n+1}(v,u+v )\nonumber\\ 
&& \times \widetilde{\mathbb{P}}(Y_{\tau^{(n+1)}}-Y_{\tau^{(n)}}\in du\Big|\tau^{(n+1)} <\infty ) \widetilde{\mathbb{P}}(Y_{\tau^{(n)}}\in dv|\tau^{(n+1)} <\infty )\Big)\widetilde{\rho}^{n+1}(1-\widetilde{\theta}(x))\nonumber\\
\end{eqnarray}

\begin{eqnarray}
&=&\phi(q,w,x) +\sum_{n=1}^{\infty}  \int_{(0,\infty)}\int_{(0,\infty)}e^{\Phi(q)(u+v)} F_{n+1}( v,u+v )(\widetilde{H}\ast \widetilde{G})(du)\nonumber \\ 
&&(\widetilde{H}^{\ast(n)}\ast \widetilde{G}^{(\ast(n)}\ast \widetilde{T}_x)(dv)\nonumber\\
\end{eqnarray}

where in the last equality we have used (\ref{step1}) and (\ref{step2}).
\end{proof}

\begin{rem}
Let $F=(F_n)_{n\geq 1}$ be a sequence of bounded measurable functions.

If $F$  satisfies  $F_1(u,v)=w(x-u, v-x)$ and  $F_n=0$ for $n\geq 2$, where $w$ is a  measurable function such that $w(\cdot,0)=0$, then $P(F,q,x)$ reduces to the classical EDPF defined by (\ref{G-S}) and then,
\begin{eqnarray}
 P(F,q,x)&=&\phi(w,q,x)\nonumber \\
         &=&\mathbb{E}\Big[ e^{-q\tau_x}w( x-Y_{\tau_x^-},Y_{\tau_x}-x ) ; \tau_x < \infty  \Big],
\end{eqnarray}
which is completely characterized in terms of Lévy measure and scale function in Lemma \ref{lem1} and Remark \ref{rem1}.
\end{rem}

\section{Capital injections}
In this subsection, we introduce the Expected Discounted Value of Capital Injections (EDVCI), which are necessary to keep the reserve process $R$ above $0$. In our context, if $R$ goes under $0$ by jumping, we must apply control to prevent the process staying in $(-\infty, 0)$. In fact, we should inject capital only when the risk process becomes negative and only when, the new record infimum under $0$ (undershoot) is reached by a jump of a subordinator.

Recall that in Theorem \ref{characterization} we have identified the extended EDPF $P(F;q;x)$ defined by (\ref{penality function}) in terms of the $q$-scale function, Lévy measure, and the densities $\widetilde{G}$ and $\widetilde{H}$ introduced in Section \ref{preliminary result}. Using the connection with the result of Theorem \ref{characterization}, we will characterize the EDVCI for the risk process defined in (\ref{riskmodel3}). We will study more explicitly the classical case driven by the Cramér-Lundberg risk model [see Einsenberg and Schmidli (2011)].


\subsection{Expected discounted value of capital injections (EDVCI)}\label{EDVCI def}
Recall that $R$ is the risk process defined in (\ref{riskmodel3}) by
\begin{equation}
R_t := x - Y_t \;, \qquad t \geq 0 \;,	
\end{equation} 
where $Y_t=-ct+ S_t -\sigma B_t$.

We denote by $C_t$ the cumulative capital injections up to time $t$. The controlled risk process $R^C$ is given by
\begin{equation}\label{Rc}
R^C_t := x - Y_t + C_t\;, \qquad t \geq 0 \;.	
\end{equation}  
We have to inject the first capital when the risk process falls below zero. Let 
$$\tau^{(1)}=\tau_x=\inf \{ s>0, R_s< 0\}$$
denote the time of first ruin, and
$$ C_1=R_{\tau_1}-x$$
denote the first injection. At time $\tau^{(1)}$ the controlled risk process $R^C$ starts with initial capital equal to zero.

For $n \geq 1$, let
$$\tau^{(n+1)}=\inf \{ t \geq \tau^{(n)}, \underline{R}_{t^-}<R_t \}$$
denote the time of the $n$-th injection. The size of the  $n$-th injection becomes
\begin{eqnarray}
C_n &=& R_{\tau^{(n)}}-R_{\tau^{(n+1)}} \nonumber\\
&=&Y_{\tau^{(n+1)}}-Y_{\tau^{(n)}}, \text{ for } n\geq 1.
\end{eqnarray}
The accumulated injections can be described as
$$C_t=\sum_{i=n}^N C_n 1_{\{\tau^{(n)}\leq t\}}.$$
Let us define the expected discounted value of capital injections (EDVCI) as
\begin{equation}\label{EDVCI}
V(q,x)=\mathbb{E}\Big[ \sum_{n=0}^N  e^{-q\tau^{(n+1)}} C_{n} \Big].
\end{equation}
Let introduce $\kappa(q,x)$ as
\begin{align}\label{k}
\kappa(q,x)=\mathbb{E}\Big[ e^{-q\tau_x}; \tau_x < \infty  \Big],
\end{align}
and
\begin{align}\label{v}
\varphi(q,x)=\mathbb{E}\Big[ e^{-q\tau_x}( Y_{\tau_x}-x ) ; \tau_x < \infty  \Big].
\end{align}
From Lemma \ref{lem1}, we can give an expression for (\ref{k}) and (\ref{v}) in terms of $q$-scale function and then
\begin{eqnarray}\label{v1}
 \varphi(q,x)&=&f_1\ast h(x)
\end{eqnarray}
and 
\begin{eqnarray}\label{k1}
 \kappa(q,x)&=&f_1\ast t(x);
\end{eqnarray}
where 
\begin{eqnarray}\label{h}
h(x)&=&e^{\Phi(q)x}\int_x^{\infty}e^{-\Phi(q)v}\int_{(0,\infty)}u\nu(du+v) dv\nonumber\\
&=&e^{\Phi(q)x}\int_x^{\infty}e^{-\Phi(q)v}\int_{(v,\infty)}(u-v)\nu(du) dv,
\end{eqnarray}
and
\begin{equation}\label{t}
t(x)=e^{\Phi(q)x}\int_x^{\infty}e^{-\Phi(q)v}\nu(v,\infty) dv.
\end{equation}

Recall that in Theorem \ref{characterization} we have identified the extended EDPF introduced in Definition \ref{penality function} in terms of the $q$-scale function, Lévy measure and notions introduced in Section \ref{preliminary result} for the model (\ref{riskmodel3}). By using connection with Theorem \ref{characterization}, the following theorem is an explicit characterization of the EDVCI defined by (\ref{EDVCI}). This would be the second main contribution of this paper that extends similar results in Einsenberg and Schmidli (2011). 

\begin{thm} \label{CI}
The Expected Discounted Value of Capital Injections EDVCI introduced in (\ref{EDVCI}) is given by
\begin{equation}\label{V}
V(q,x)=\varphi(q,x) + \frac{\delta(q, \sigma)}{1-\xi(q, \sigma)}\kappa(x,\sigma)\;, 
\end{equation}
\end{thm}
where $\xi(q, \sigma)$ and $\delta(q,\sigma)$ are given, respectively, by
\begin{equation}
\xi(q, \sigma)= (1+ \frac{\Phi(q)\sigma^2}{2c + \Phi(q)\sigma^2})\big[1-\frac{q+\frac{\sigma^2}{2} \Phi(q)}{\Phi(q)(c+\frac{\sigma^2}{2}  \Phi(q))} \big]\;,
\end{equation}
and
\begin{equation}
\delta(q,\sigma)= \frac{2c}{\Phi(q)(2c+  \Phi(q)\sigma^2)}\Big[\frac{2q}{\Phi(q)(2c + \Phi(q)\sigma^2)}+\rho -1 \Big].
\end{equation}
Equation (\ref{V}) gives an explicit formula of the expected value which should be injected at each deficit time at and after ruin that will allow the insurance company to survive and continue its operations when the risk process continuous to jump downwards.

\begin{proof}
Let us consider the sequence of functions $F=(F_n)_{n\geq 1}$  as defined above. Since we suppose $F_1(v,u)=u-x$ and $F_n(v,u)=u-v$ for $u \geq0,\quad v \in \mathbb{R}$ and $n\geq 2$, then the extended EDPF associated with $F$ and $q$, $P(F,q,x)$, defined in (\ref{penality function}), is equal to $V(q,x)$. By using Theorem \ref{characterization}, we can easly derive (\ref{V}) and then
\begin{eqnarray}\label{id1}
V(q,x)&=&\mathbb{E}\Big[ e^{-q\tau_x}( Y_{\tau_x}-x ) ; \tau_x < \infty  \Big]+\nonumber \\&&\sum_{n=1}^{\infty} \int_{(x,\infty)}\int_{(0,\infty)}e^{\Phi(q)(u+v)} u  (\widetilde{H}\ast \widetilde{G})(du)(\widetilde{H}^{\ast n}\ast \widetilde{G}^{\ast n}\ast T^{(q)}_x)(dv)\nonumber\\
&=&\varphi(q,x)+\int_{(0,\infty)}e^{\Phi(q)u} u\;\widetilde{H}\ast \widetilde{G}(du)  \nonumber\\ &&\sum_{n=0}^{\infty} \int_{(x,\infty)}e^{\Phi(q)v}\widetilde{H}^{\ast n}\ast \widetilde{G}^{\ast n}\ast \widetilde{T}_x(dv)\nonumber\\
&=&\varphi(q,x)+\int_{(0,\infty)}e^{\Phi(q)u} u\;\widetilde{H}\ast \widetilde{G}(du)  \; \sum_{n=0}^{\infty} \int_{(x,\infty)}e^{\Phi(q)v}\widetilde{T}_x(dv) \nonumber \\
&& \int_{(0,\infty)}e^{\Phi(q)v}\widetilde{H}^{\ast n}\ast \widetilde{G}^{\ast n}(dv) \nonumber \\
&=&\varphi(q,x)+\underbrace{\int_{(0,\infty)}e^{\Phi(q)u} u\;\widetilde{H}\ast \widetilde{G}(du)}_{I}  \; \sum_{n=0}^{\infty} \underbrace{\int_{(x,\infty)}e^{\Phi(q)v}\widetilde{T}_x(dv)}_{II} \nonumber \\
 &&\big[\underbrace{\int_{(0,\infty)}e^{\Phi(q)v}\widetilde{H} \ast \widetilde{G}(dv)}_{III}\big]^n. \nonumber \\
\end{eqnarray}

Recall that  $\widetilde{\mathbb{P}}(Y_{\tau_x}\in dv\; \tau_x < \infty)=\widetilde{T}_x(dv)$, for $v>x$,  hence by (\ref{change measure}), (II) is equal to 
\begin{eqnarray}
\widetilde{\mathbb{E}}\big[e^{\Phi(q)Y_{\tau_x}};\; \tau_x < \infty] &=& \mathbb{E}\big[e^{-q\tau_x};\; \tau_x < \infty]\nonumber \\
&=&\kappa(q,x).
\end{eqnarray}
(III) is equal to 
\begin{equation}\label{egal}
\int_{(0,\infty)}e^{\Phi(q)u} \;\widetilde{H}\ast \widetilde{G}(du) =\int_{(0,\infty)}e^{\Phi(q)u} \;\widetilde{H}(du)\;\int_{(0,\infty)}e^{\Phi(q)u} \; \widetilde{G}(du)\;, 
\end{equation}
where 
\begin{eqnarray}
\int_{(0,\infty)}e^{\Phi(q)u} \;\widetilde{H}(du)&=&\int_{(0,\infty)}\frac{1}{c+\Phi(q)\sigma^2}\int_0^{\infty}e^{-\Phi(q)y} \;\nu(du+y)dy\nonumber\\
&=&\frac{1}{c+\Phi(q)\sigma^2}\int_0^{\infty}e^{-\Phi(q)y} \;\nu(y,\;\infty)dy.
\end{eqnarray}
Recall that 
\begin{equation}\label{Phi}
\psi(\Phi(q))=q \;,
\end{equation}
 by using  integration by part, (\ref{Phi}) is equivalent to
\begin{equation}
c\Phi(q) + \frac{\sigma^2\Phi(q)^2}{2}-\Phi(q)\int_0^{\infty}e^{-\Phi(q)y} \;\nu(y,\;\infty)dy =q
\end{equation}
and then,
\begin{equation}\label{egal1}
\int_0^{\infty}e^{-\Phi(q)y} \;\nu(y,\infty)dy = c + \frac{\sigma^2 \Phi(q)}{2}-\frac{q}{\phi(q)}.
\end{equation}
We have
\begin{eqnarray}
\int_{(0,\infty)}e^{\Phi(q)u} \;\widetilde{G}(du)&=&\int_{(0,\infty)}\frac{2\widetilde{c}}{\sigma^2}e^{-(\frac{2\widetilde{c}}{\sigma^2}-\Phi(q))y} dy\nonumber\\
&=&1+ \frac{\phi(q)\sigma^2}{2c+\phi(q)\sigma^2}=\alpha(q,\sigma).
\end{eqnarray}
By substituting (\ref{egal1}) in (\ref{egal}), we conclude that 
\begin{eqnarray}
\int_{(0,\infty)}e^{\Phi(q)u} \;\widetilde{H}\ast \widetilde{G}(du)&=&\frac{\alpha(q,\sigma)}{c+\sigma^2 \Phi(q)}\big(c+\sigma^2 \Phi(q)-\frac{q}{\Phi(q)} \big)\nonumber\\
&=&\alpha(q,\sigma)\big[1-\frac{q+\frac{\sigma^2}{2} \Phi(q)}{\Phi(q)(c+\frac{\sigma^2}{2}  \Phi(q))} \big]\nonumber\\
&=&\xi(q, \sigma).
\end{eqnarray}

(I) is equal to 
\begin{eqnarray}
\int_{(0,\infty)}e^{\Phi(q)u} u\;\widetilde{H}\ast \widetilde{G}(du)&=&\int_{(0,\infty)}e^{\Phi(q)u} \big[\int_{(0,\infty)}(u+v)e^{\Phi(q)v} \widetilde{G}(dv)\big] \;\widetilde{H}(du)\nonumber\\
&=&\int_{(0,\infty)}e^{\Phi(q)u} \big[\int_{(0,\infty)}(u+v)e^{\Phi(q)v} \widetilde{G}(dv)\big] \;\widetilde{H}(du)
\nonumber\\
&=&\int_{(0,\infty)}e^{\Phi(q)u} \big[u\int_0^{\infty}\frac{2\widetilde{c}}{\sigma^2}e^{-(\frac{2\widetilde{c}}{\sigma^2}-\Phi(q))v} dv \nonumber\\&&+\int_0^{\infty}\frac{2\widetilde{c}}{\sigma^2}ve^{-(\frac{2\widetilde{c}}{\sigma^2}-\Phi(q))v} dv\big] \;\widetilde{H}(du)
\nonumber\\
&=&\int_{(0,\infty)}e^{\Phi(q)u} \big[\alpha(q,\sigma)(u +\frac{\sigma^2}{2c + \Phi(q)\sigma^2})\big] \;\widetilde{H}(du)
\nonumber\\
&=&\alpha(q,\sigma)\frac{1}{c+\Phi(q)\sigma^2}\Big[\frac{\sigma^2}{2c + \Phi(q)\sigma^2}\int_{(0,\infty)}e^{-\Phi(q)y}\;\nu(y, \infty)dy \nonumber \\&&+\int_0^{\infty}\int_{(0,\infty)}ue^{-\Phi(q)y}\;\nu(du +y)dy\Big]\nonumber\\
&=&\alpha(q,\sigma)\frac{1}{c+\Phi(q)\sigma^2}\Big[\frac{\sigma^2}{2c + \Phi(q)\sigma^2}\int_{(0,\infty)}e^{-\Phi(q)y}\;\nu(y, \infty)dy \nonumber\\&&+\frac{1}{\Phi(q)}\big[\int_0^{\infty}\nu(y, \infty)dy - \int_0^{\infty}e^{-\Phi(q)y}\pi(y, \infty)dy\big]\Big]\nonumber\\
&=&\alpha(q,\sigma)\frac{1}{c+\Phi(q)\sigma^2}\Big[\big(\frac{\sigma^2}{2c + \Phi(q)\sigma^2}-\frac{1}{\Phi(q)}\big)\nonumber \\
&& \int_{(0,\infty)}e^{-\Phi(q)y}\;\nu(y,\infty)dy +\frac{1}{\Phi(q)}\mathbb{E}[S_1]\Big].\nonumber\\
\end{eqnarray}
Using (\ref{egal1}), the last equality can be written as 
\begin{eqnarray}
&&\alpha(q,\sigma)\frac{1}{c+\Phi(q)\sigma^2}\Big[\frac{-2c}{\Phi(q)(2c + \Phi(q)\sigma^2)}\big(c + \frac{\sigma^2 \Phi(q)}{2}-\frac{q}{\phi(q)} \big) +\frac{c}{\Phi(q)}\rho \Big]\nonumber\\
&&=\frac{\alpha(q,\sigma)c}{\Phi(q)(c+\Phi(q)\sigma^2)}\Big[\frac{2q}{\Phi(q)(2c + \Phi(q)\sigma^2)}+\rho -1 \Big]\nonumber\\
&&=\delta(q,\sigma)
\end{eqnarray}
and then (I) is equal to
\begin{equation}
 \int_{(0,\infty)}e^{\Phi(q)u} u\;\widetilde{H}\ast \widetilde{G}(du)=\delta(q,\sigma).
\end{equation}
In addition, by using the identification above of (II) and (III), (\ref{id1}) is equal to 
\begin{eqnarray}\label{id2}
&&\varphi(q,x)+\delta(q, \sigma)  \; \sum_{n=0}^{\infty} \kappa(x,\sigma)\xi(q, \sigma)^n \nonumber \\
&&=\varphi(q,x) + \frac{\delta(q, \sigma)}{1-\xi(q, \sigma)}\kappa(x,\sigma)
\end{eqnarray} 
and the theorem follows. 
\end{proof}\\

In the next subsection, we illustrate the previous result by a specific example of Cramér-Lundenberg risk model when the Brownian component vanishes.
\subsection{Classical risk model}

Let us consider the classical risk model. The number of claims is assumed to follow a Poisson process $(N_t)_{t\geq 0}$ with intensity $\lambda$. We denote by $(Z_n)_{n\geq 1}$ the claim sizes which are independant of $(N_t)_{t\geq 0}$, positive and \emph{iid} with distribution function $K$ and first moment $\mu$. The aggregate claim process is given by
\begin{equation}\label{classrisk}
 R_t=x +ct-\sum_{i=1}^{N_t}Z_i,
\end{equation}
where $x$ is the initial capital and $c$ is the premium rate.

The surplus process introduced by (\ref{Rc}) then has the form
\begin{equation}
R^C_t := x+ct - \sum_{i=1}^{N_t}Z_i + C_t\;, \qquad t \geq 0 \;.	
\end{equation}
Note that $S_t=\sum_{i=1}^{N_t}Z_i$ and then $\mathbb{E}[S_1]=\lambda \mu$. However, the net profit condition (\ref{netprofit}) becomes $c>\lambda \mu$ and then $\rho=\lambda \mu /c$.
Since $\sigma=0$, it follows from Theorem \ref{CI} that the EDVCI for classical model introduced above is given by
\begin{eqnarray} \label{classV}
V(q,x)&=&\varphi(q,x) + \frac{\delta(q,0)}{1-\xi(q, 0)}\kappa(x,0),
\end{eqnarray}
where
\begin{align}
\xi(q,0)=1-\frac{q}{\Phi(q)c} , \;\text{ and  }
\delta(q,0)=\frac{1}{\Phi(q)}\Big[\frac{q}{\Phi(q)c }+\rho -1 \Big]=\frac{q-(c-\lambda \mu)\Phi(q)}{c\Phi(q)^2}.
\end{align}
Since the Brownian component vanishes, $Y$ is a process with bounded variation and then, from Kyprianou (2006),
$W_{\Phi(q)}(dx)=\sum_{n=0}^{\infty}\eta^{\ast n}(dx)$, where 
\begin{eqnarray}
\eta(dx)&= &\frac{1}{c}\widetilde{\nu}(x, \infty)dx\nonumber \\
&=&\frac{1}{c}\int_{(x, \infty)}e^{-\Phi(q)u}\nu(du)\;dx.
\end{eqnarray}
Note that Equation (\ref{f1}) reduces to
\begin{eqnarray}\label{f22}
f_1(x)&=&\frac{\lambda}{c}\sum_{n=0}^{\infty}\int_{(x, \infty)}e^{-\Phi(q)u}K(du),
\end{eqnarray}
when in the last equality, we have used the identity  $\nu(du)=\lambda K(du)$. In addition Equations (\ref{h}) and (\ref{t}) reduce to
\begin{eqnarray}\label{h1}
h(x)&=&\frac{\lambda}{c}e^{\Phi(q)x}\int_x^{\infty}e^{-\Phi(q)v}\int_{(v,\infty)}(u-v)K(du) dv
\end{eqnarray}
and 
\begin{equation}\label{t1}
 t(x)=\frac{\lambda}{c}e^{\Phi(q)x}\int_x^{\infty}e^{-\Phi(q)v}(1-F(v)) dv.
\end{equation}
Consequently, Equation \ref{classV} reduces to the expression of the EDVCI for classical model (\ref{classrisk}) given in Einsenberg and Schmidli (2011).
\section{Conclusion}
In this paper we have generalized the Expected Discounted Penality Function (EDPF) indroduced by Gerber and Shiu (1997, 1998) to include the successive minima reached by the risk process because claim after ruin. In addition to the surplus before ruin and the deficit at ruin, we have added to the EDPF the expectation of a sequence of discounted functions of minima in the context of subordinator risk model perturbed by a Brownian motion. By using some results in Huzak \emph{et al.} (2004) and developments in theory of fluctuations for spectrally negative Lévy processes, we have derived an expicit expression of this extended EDPF.

Our generalization of the EDPF includes information on the path behavior of the risk process not only in a neighborhood of the ruin time, but also after ruin. Such information is relevant for risk management process aimed at preventing successive occurences of the insolvency events. In addition to the classical EDPF introduced by Gerber and Shiu (1997, 1998), the new extension of EDPF contains a sequence of expected discounted functions of successive minima reached by jumps after ruin. This sequence of EDPF has many interesting potential applications. For example, it could be used as a predictive tool for successive deficit times after ruin. In particular, we have used this extended EDPF to derive explicitly the Expected Discounted Value of Capital Injections EDVCI which are necessary to keep the risk process above zero.
%

Inspired by results of Huzak \emph{et al.} (2004) and developpements in fluctuation theory for spectrally negative Lévy processes, we provide a characterization for this extended EDPF in a setting involving a cumulative claims modelled by a subordinator, and spectrally negative perturbation. We illustrate how the ESDPF can be used to compute the expected discounted value of capital injections  (EDVCI) for Brownian perturbed risk model. 
The main contributions of this paper are found in Theorems \ref{characterization} and \ref{CI}. These two results give expressions for the extended EDPF given by Definition \ref{penality function}, and the EDVCI introduced in Subsection \ref{EDVCI def}. These expressions can be easily computed by considering particular examples of subordinators. 

In oder to make these results as explicit as possible, we have used the risk model driven by Brownian perturbed subordinator process. Further work is needed in order to give these results for more general risk processes driven by a spectrally negative Lévy process.

\hspace{.4cm}


\vspace{.3cm}
\newpage
{\small
\begin{thebibliography}{}

\bibitem{Asm} Asmussen, S.~(2003). \textit{Applied Probability and Queues, 2nd edition}. Springer-Verlag, Berlin, Heidelberg, New York.
\bibitem{Asm1} Asmussen, S. and Albrecher, H.~(2010). \textit{Ruin Probabilities}. World Scientific Publishing, London.
\bibitem{Bertoin} Bertoin, J.~(1996). \textit{Lévy processes}. Combridge university press.
\bibitem{BM2} Ben Salah, Z. and Morales, M.~(2012). \textit{L\'{e}vy Systems and the Time Value of Ruin for Markov Additive Processes }. European Actuarial Journal, \textbf{2(2)}, 289-317.
\bibitem{Biffis} Biffis,E. and Kyprianou, A.~(2010). \textit{A note on scale function and the time value of ruin}. Insurrance: Mathematics and Economics. \textbf{46}, 85-91.
\bibitem{BiffisMorales} Biffis,E. and Morales, M.~(2010). \textit{On a generalization  of the Gerber-Shiu function to path-dependent penalities}. Insurance: Mathematics and Economics. \textbf{46}, 92-97.
\bibitem{Duf} Dufresne, F. and Gerber, H.~(1991). \textit{Risk theory for compounted Poisson process that is perturbed by diffusion}. Insurance Math.Econom.\textbf{10}, 51-59.
\bibitem{Duf G} Dufresne, F., Gerber and Shiu, E. W.~(1991). \textit{Risk theory with gamma process}. Astin Bull.\textbf{21} 177-192.
\bibitem{Doney Kyprianou} Doney, R. and Kyprianou, A.~(2006). \textit{Overshoots and undershoots of Lévy processes}. Ann. Appl. Probab. \textbf{16(1)}, 91-106.
\bibitem{Schmidli} Einsenberg, I. and Schmidli, H.~(2011). \textit{Minimising expected discounted capital injections by reinsurance in a classical risk model}. Scand.Actuarial J. \textbf{3}, 155-176.
\bibitem{Huzak} Huzak, M., Perman, M., Sikic, H. and Vondracek, Z..~(2004). \textit{Ruin probabilities and decompositions for general perturbed risk processes}. The Annals of Applied probability. \textbf{14(3)}, 1378-1397.
\bibitem{Furrer} Furrer, H.~(1998). \textit{Risk processes perturbed by $\alpha$-stable Lévy motion.}. Scand. Actuar. J. \textbf{1}, 59-74.
\bibitem{Garrido} Garrido, J. and Morales, M.~(2006). \textit{On the expected penality function for Lévy risk process.}. N. Am. Actuar. J. \textbf{10(4)}, 196-218.
\bibitem{GS} Gerber, H. and Shiu, E.~(1997). \textit{The joint distribution of the time of ruin, the surplus immediately before ruin, and the deficit at ruin}. Insur. Math. Econ. \textbf{21}, 129-137.
\bibitem{GS2} Gerber, H. and Shiu, E.~(1998). \textit{On the time value of ruin}. North Am. Actuar. J. \textbf{2(1)}, 48-78.
\bibitem{GS4} Gerber, H. and Shiu, E.~(1998b). \textit{Pricing Perpetual Options for Jump Processes}. North Am. Actuar. J \textbf{2(3)}, 101-112.
\bibitem{GS3} Gerber, H. and Landry, B.~(1998). \textit{On a discounted Penalty at Ruin in a jump-diffusion and Perpetual Put option}. Insur. Math. Econ. \textbf{22}, 263-276.
\bibitem{Kyp} Kyprianou, A.~(2006). \textit{Introductory Lectures on Lévy Processes with Aplications }. Universitext, Springer.


\end{thebibliography}
}
\end{document}